\documentclass[a4paper,12pt]{article}
\usepackage{amssymb}
\usepackage{amsthm}
\usepackage[dvips]{graphicx}
\usepackage{amsmath}
\usepackage{fancyhdr}

\theoremstyle{plain}
\newtheorem{theorem}{Theorem}[section]
\newtheorem{proposition}[theorem]{Proposition}
\newtheorem{lemma}[theorem]{Lemma}

\newtheorem{problem}{Problem\!\!}

\newtheorem{remark}{\sc Remark\!\!}

\newtheorem{example}{\sc Example\!\!}

\newtheorem{notation}{\sc Notations\!\!}

\renewcommand{\baselinestretch}{1.15}

\newenvironment{reference}[1]{%

\begin{flushleft}\normalsize{\textbf{References}}\end{flushleft}%
\begin{enumerate}\setlength{\itemsep}{-5pt}\small
}{\end{enumerate}}

\renewcommand{\section}[1]{%
\addtocounter{section}{1}\setcounter{subsection}{0}%
\setcounter{theorem}{0}%
\begin{flushleft}
\normalsize{\textbf{\arabic{section}. #1}}
\end{flushleft}}
\renewcommand{\subsection}[1]{\addtocounter{subsection}{1}%
\noindent\textsc{\arabic{section}.\arabic{subsection}. #1 }}

\title{\vspace*{-21mm}
\large{\textbf{
On the Iwasawa invariants of a link in the 3-sphere
}}
}

\author{ 
\textsc{\normalsize Teruhisa Kadokami}%
\and 
\textsc{\normalsize Yasushi Mizusawa}%
}

\date{}

\begin{document}
{\renewcommand{\baselinestretch}{1.05} \maketitle}

\vspace*{-14mm}
\renewcommand{\abstractname}{}
{\renewcommand{\baselinestretch}{1.05}
\begin{abstract}{\small 
\noindent\textsc{Abstract.} 
Based on the analogy between knots and primes, J. Hillman, D. Matei and M. Morishita defined the Iwasawa invariants for sequences of cyclic covers of links with an analogue of Iwasawa's class number formula of number fields. In this paper, we consider the existence of covers of links with prescribed Iwasawa invariants, discussing analogies in number theory. We also propose and consider a problem analogous to Greenberg's conjecture. 
}\end{abstract}}

\vspace*{1pt}

\section{Introduction}

Let $L=K_1 \cup K_2 \cup \cdots \cup K_r$ be an $r$-component link in the $3$-sphere $S^3$. Let $\ell_{ij}=\mathrm{lk}(K_i,K_j)$ be the linking number of the components $K_i$ and $K_j$ $(i\neq j)$. Let $X$ be the complement of an open tubular neighborhood of $L$, and $G_L=\pi_1(X)$ the link group of $L$. Let $\pi : \widetilde{X} \rightarrow X$ be the maximal abelian cover of $X$. Let $[m_i] \in G_L$ be the class of the meridian $m_i$ of the component $K_i$, where we regard $m_i$ as a loop from the base point $\ast$ of $X$. Then $H_1(X;\mathbb Z) \simeq \mathrm{Aut}(\widetilde{X}/X) \simeq G_L/G_L^{\prime}$ is generated by $t_i=[m_i]G_L^{\prime}$ $(i=1,2,\cdots,r)$ as a free $\mathbb Z$-module of rank $r$, where $G_L^{\prime}$ is the commutator subgroup of $G_L$. Put the Laurent polynomial ring $\varLambda_r=\mathbb Z[t_1^{\pm 1},t_2^{\pm 1},\cdots,t_r^{\pm 1}] \simeq \mathbb Z[G_L/G_L^{\prime}]$. Then the Alexander module $A_L=H_1(\widetilde{X},\pi^{-1}(\ast);\mathbb Z)$ and the link module $B_L=H_1(\widetilde{X};\mathbb Z) \simeq G_L^{\prime}/G_L^{\prime\prime}$ are finitely generated $\varLambda_r$-module, where $G_L^{\prime\prime}$ is the commutator subgroup of $G_L^{\prime}$. 

Based on the analogies between link groups and Galois groups of number fields with restricted ramification (cf.\ \cite{Mor02} \cite{Mor10} \cite{Mor} etc.), we regard $\widetilde{X}$ as an analogue of the maximal free abelian pro-$p$-extension $\widetilde{k}$ of a number field $k$, where $p$ is a fixed prime number. Then the Iwasawa theory for $\widetilde{k}/k$ is set up in parallel with Alexander-Fox theory (cf.\ e.g.\ \cite{Gre78} \cite{NQD83}). Let $S=\{\wp_1,\wp_2,\cdots,\wp_r\}$ be the set of all primes of $k$ lying over $p$, and $G_S=G_S(k)$ the Galois group of the maximal pro-$p$-extension of $k$ unramified outside $S$. Then $H_1(G_S;\mathbb Z_p) \simeq G_S/G_S^{\prime} \simeq \mathrm{Gal}(\widetilde{k}/k) \simeq \mathbb Z_p^{r_2+1+\delta}$ for a subgroup $G_S^{\prime}$ and $\delta \ge 0$, where $\mathbb Z_p$ is (the additive group of) the $p$-adic integer ring, and $r_2$ is the number of infinite complex places of $k$. Leopoldt's conjecture predicts that $\delta=0$ for any $k$ and $p$ (cf.\ \cite{Was} Theorem 13.4). For the simplicity, we assume that $\delta=0$ throughout this paper. If $k$ has the class number $h_k \not\equiv 0 \pmod{p}$, $G_S/G_S^{\prime}$ is generated by $\tau_iG_S^{\prime}$ $(i=1,2,\cdots,r_2+1)$ where each $\tau_i$ is an element of inertia group of some $\wp_j$. The $S$-ramified Iwasawa module $\mathfrak{X}_S=G_S^{\prime}/G_S^{\prime\prime}$ is a finitely generated $\mathbb Z_p[[G_S/G_S^{\prime}]]$-module, where $G_S^{\prime\prime}$ is the commutator subgroup of $G_S^{\prime}$. 
\[
\begin{array}{ccc}
\hbox{link }L=K_1 \cup K_2 \cup \cdots \cup K_r &\longleftrightarrow& \hbox{primes }S=\{\wp_1,\wp_2,\cdots,\wp_r\} \\
\hbox{link group }G_L &\longleftrightarrow& \hbox{Galois group }G_S \\
G_L/G_L^{\prime} \simeq \mathrm{Aut}(\widetilde{X}/X) \simeq \mathbb Z^r &\longleftrightarrow& G_S/G_S^{\prime} \simeq \mathrm{Gal}(\widetilde{k}/k) \simeq \mathbb Z_p^{r_2+1} \\
\hbox{meridians }[m_i] &\longleftrightarrow& \hbox{inertia groups }\langle\,\cdots,\tau_i,\cdots\rangle \\
\hbox{link module }B_L \simeq G_L^{\prime}/G_L^{\prime\prime} &\longleftrightarrow& \hbox{Iwasawa module }\mathfrak{X}_S=G_S^{\prime}/G_S^{\prime\prime} \\
\end{array}
\]

For a surjective homomorphism $\sigma_{\mathbf{z}} : H_1(X;\mathbb Z) \rightarrow \mathbb Z$ (resp.\ $G_S/G_S^{\prime} \rightarrow \mathbb Z_p$), we have an infinite cyclic cover $X_{\mathbf{z}} \rightarrow X$ (resp.\ a $\mathbb Z_p$-extension $k_{\infty}/k$) corresponding to the kernel. Iwasawa \cite{Iwa59} gave a formula for the growth of the $p$-parts of the class numbers in the $\mathbb Z_p$-extension $k_{\infty}/k$ (cf.\ \S2). Hillman, Matei and Morishita \cite{HMM} (and \cite{KM08}) gave an analogous formula for the growth of the $p$-parts of the order of $1$-homology groups in the tower of $p^n$-fold cyclic branched covering spaces of $L$ along $X_{\mathbf{z}}$ (cf.\ \S2 Theorem \ref{Iwasawa type formula} (2)). These formulas are described by the Iwasawa invariants $\lambda$, $\mu$ and $\nu$, which are closely related to the structure of link modules or Iwasawa modules. 

In this paper, we show some basic properties of Iwasawa invariants of $L$, discussing analogies with Iwasawa theory of number fields. In \S2 and \S3, we decide the parity of $\lambda$-invariants of $L$ and the existence of $L$ with prescribed $\lambda$- and $\mu$-invariants. We also give some examples of $2$-component links $L$ with explicit Iwasawa invariants. In \S4 and \S5, we propose and consider an analogous problem to Greenberg's conjecture which asserts that the maximal unramified quotient of $\mathfrak{X}_S$ is {\em pseudonull}. 

\begin{notation}{\em 
For a noetherian unique factorization domain $\varLambda$ and a finitely generated $\varLambda$-module $A$, we denote by $\mathrm{rank}_{\varLambda}A$ the free rank of $A$. The {\em divisorial hull} of an ideal $E$ is the intersection of all principal ideals of $\varLambda$ containing $E$ (cf.\ \cite{Hil} p.49). Let $E_{\varLambda}^{(i)}(A)$ be the $i$th elementary ideal of $A$, and $\varDelta_{\varLambda}^{(i)}(A)$ a generator of the divisorial hull of $E_{\varLambda}^{(i)}(A)$ as a principal ideal. A finitely generated torsion $\varLambda$-module $A$ is {\em pseudonull} if and only if the divisorial hull of the annihilator ideal $\mathrm{Ann}_{\varLambda}A$ is $\varLambda$ (cf.\ \cite{Hil} Theorem 3.5), i.e., $\mathrm{Ann}_{\varLambda}A$ has at least two relatively prime elements. We denote by $|A|$ the order of a $\mathbb Z$-module $A$. Note that $|A|=0$ if $A$ is infinite. Let $v_p$ be the additive $p$-adic valuation normalized as $v_p(p)=1$, and $\varPhi_{p^n}(t)=\sum_{i=0}^{p-1}(t^{p^{n-1}})^i=(t^{p^n}-1)/(t^{p^{n-1}}-1)$ the $p^n$th cyclotomic polynomial for $n \ge 1$. 
}\end{notation}

\section{Iwasawa invariants and Alexander polynomials}

For $f(t) \in \varLambda_1=\mathbb Z[t^{\pm 1}]$, we denote by $\deg^{\prime} f(t)$ the reduced degree as a Laurent polynomial. The Alexander polynomial of $L$ is defined as 
\[
\varDelta_L=\varDelta_L(t_1,t_2,\cdots,t_r)=\varDelta_{\varLambda_r}^{(1)}(A_L)=\varDelta_{\varLambda_r}^{(0)}(B_L)
\]
up to multiplication by units in $\varLambda_r$ (cf.\ \cite{Kaw96} Proposition 7.3.4 (2)). For $\mathbf{z}=(z_1,z_2,\cdots,z_r) \in \mathbb Z^r$ such that $\gcd(z_1,z_2,\cdots,z_r)=1$ and $\prod_{i=1}^{r} z_i \neq 0$, we define a surjective homomorphism 
\[
\sigma_{\mathbf{z}} : H_1(X;\mathbb Z) \rightarrow \mathbb Z : t_i \mapsto z_i . 
\]
Then $\mathrm{Ker}\,\sigma_{\mathbf{z}}$ corresponds to an infinite cyclic cover $X_{\mathbf{z}} \rightarrow X$, and the reduced Alexander polynomial associated to $X_{\mathbf{z}}$ is defined in $\varLambda_1 \simeq \mathbb Z[\mathrm{Aut}(X_{\mathbf{z}}/X)]$ as 
\[
\varDelta_{L,\mathbf{z}}(t) = \varDelta_{\varLambda_1}^{(0)}(H_1(X_{\mathbf{z}};\mathbb Z)) = (t-1)\varDelta_L(t^{z_1},t^{z_2},\cdots,t^{z_r}) 
\]
if $r \ge 2$, and $\varDelta_{L,\mathbf{z}}(t)=\varDelta_L(t)$ if $r=1$ (cf.\ \cite{Kaw96} Proposition 7.3.10 (1)). For $\mathbf{1}=(1,1,\cdots,1)$, $X_{\mathbf{1}}$ is called the total linking number covering space of $X$. 
Let $M_{\mathbf{z},p^n} \rightarrow S^3$ be the Fox completion (cf.\ e.g.\ \cite{Mor}) of the $p^n$-fold subcover $X_{\mathbf{z},p^n} \rightarrow X$ of $X_{\mathbf{z}} \rightarrow X$. Then $M_{\mathbf{z},p^n} \rightarrow S^3$ is a cyclic branched cover of $L$ associated to $X_{\mathbf{z},p^n} \rightarrow X$. 

For a fixed prime number $p$, we put $\widehat{\varLambda}_1 = \mathbb Z_p[[T]]$ the ring of formal power series. We can regard $\varLambda_1$ as a subring of $\widehat{\varLambda}_1$ by identifying $t=1+T$. If $\varDelta_{L,\mathbf{z}}(t) \neq 0$, $\varDelta_{L,\mathbf{z}}(1+T) \in \widehat{\varLambda}_1$ is uniquely written as 
\[
\varDelta_{L,\mathbf{z}}(1+T)=p^{\mu_{L,\mathbf{z}}} P_{L,\mathbf{z}}(T) U(T)
\]
by the $p$-adic Weierstrass preparation theorem (\cite{Was} Theorem 7.3), where $0 \le \mu_{L,\mathbf{z}} \in \mathbb Z$, $U(T) \in \widehat{\varLambda}_1^{\times}$ and $P_{L,\mathbf{z}}(T)$ is a distinguished polynomial of degree $\lambda_{L,\mathbf{z}}=\deg P_{L,\mathbf{z}}(T) \le \deg^{\prime} \varDelta_{L,\mathbf{z}}(t)$, i.e., a monic polynomial such that $P_{L,\mathbf{z}}(T) \equiv T^{\lambda_{L,\mathbf{z}}} \pmod{p}$. Further, if $\varDelta_{L,\mathbf{z}}(t) \in \mathbb Z[t]$, then $U(T) \in \mathbb Z_p[T]$ (\cite{Was} Lemma 7.5). 

As pointed out by Mazur \cite{Maz}, the objects analogous to these invariants are as follows. The $\mathbb Z_p$-extension $k_{\infty}/k$ is regarded as a tower of cyclic subextensions $k_n/k$ of degree $p^n$. The Galois group $Y_{k_{\infty}}$ of the maximal unramified abelian pro-$p$-extension over $k_{\infty}$ is a module over $\widehat{\varLambda}_1 \simeq \mathbb Z_p[[\mathrm{Gal}(k_{\infty}/k)]] : 1+T \leftrightarrow \gamma$, where $\mathrm{Gal}(k_{\infty}/k)=\gamma^{\mathbb Z_p}$. The Iwasawa polynomial is defined as $p^{\mu_{k_{\infty}}} P_{k_{\infty}}(T) = \varDelta_{\widehat{\varLambda}_1}^{(0)}(Y_{k_{\infty}})$ where $P_{k_{\infty}}(T)$ is a distinguished polynomial of degree $\lambda_{k_{\infty}}$. Then there is an integer $\nu_{k_{\infty}}$ satisfying Iwasawa's class number formula (cf.\ e.g.\ \cite{Was}); 
\[
v_p(|Cl(k_n)|) = \lambda_{k_{\infty}} n + \mu_{k_{\infty}} p^n + \nu_{k_{\infty}}
\]
for all sufficiently large $n$, where $Cl(k)$ denotes the ideal class group of a number field $k$ which is analogous to $H_1(M;\mathbb Z)$ of a rational homology $3$-sphere $M$. 
\[
\begin{array}{ccc}
\hbox{infinite cyclic cover }X_{\mathbf{z}} \rightarrow X &\longleftrightarrow& \hbox{$\mathbb Z_p$-extension }k_{\infty}/k \\
\cdots \rightarrow M_{\mathbf{z},p^n} \rightarrow \cdots \rightarrow S^3 &\longleftrightarrow& k \subset \cdots \subset k_n \subset \cdots \subset k_{\infty} \\\hbox{Alexander polynomial }\varDelta_{\varLambda_1}^{(0)}(H_1(X_{\mathbf{z}};\mathbb Z)) &\longleftrightarrow& \hbox{Iwasawa polynomial }\varDelta_{\widehat{\varLambda}_1}^{(0)}(Y_{k_{\infty}}) \\
\end{array}
\]

As an analogy of Iwasawa's formula, we have obtained the following theorem (cf.\ \cite{HMM} Theorem 5.1.7, \cite{KM08} Theorem 2.1 and its proof). 

\begin{theorem}\label{Iwasawa type formula}
{\em (1)} Put $v=\max\{v_p(z_i)\,|\,1 \le i \le r\,\}$. For all $n \ge v$, we have 
\[
|H_1(M_{\mathbf{z},p^n};\mathbb Z)|= |H_1(M_{\mathbf{z},p^v};\mathbb Z)| \, \Big|\! \prod_{\zeta^{p^n}=1 \atop \zeta^{p^v} \neq 1} \varDelta_{L,\mathbf{z}}(\zeta) \Big|. 
\]
{\em (2)} If $|H_1(M_{\mathbf{z},p^n};\mathbb Z)| \neq 0$ for any $n$, there is an integer $\nu_{L,\mathbf{z}}$ such that 
\[
v_p(|H_1(M_{\mathbf{z},p^n};\mathbb Z)|) = \lambda_{L,\mathbf{z}} n + \mu_{L,\mathbf{z}} p^n + \nu_{L,\mathbf{z}}
\]
for all sufficiently large $n$. 
\end{theorem}

\begin{remark}{\em 
The formula of (1) is a direct consequence of Mayberry-Murasugi's formula (cf.\ \cite{MM82} \cite{Por04}). The formula of (2) is induced from (1) by similar arguments to the proof of Theorem 5.1.7 of \cite{HMM}. 
}\end{remark}

By the basic properties of Alexander polynomials, we obtain the following theorem, which decides the parity of $\lambda_{L,\mathbf{z}}$ and the existence of $L$ with prescribed $\lambda_{L,\mathbf{1}}$ and $\mu_{L,\mathbf{1}}$. 

\begin{theorem}\label{mainthm}
{\em (i)} If $r=1$, then $\lambda_{L,\mathbf{z}}=\mu_{L,\mathbf{z}}=0$ and $|H_1(M_{\mathbf{z},p^n};\mathbb Z)| \neq 0$ for any $n$. 

{\em (ii)} If $r \ge 2$, then $\lambda_{L,\mathbf{z}} \ge r-1$ and 
\[
\lambda_{L,\mathbf{z}} \equiv \bigg\{
\begin{array}{ll}
r-1 \pmod{2} & \hbox{if $p \neq 2$} \\
\deg^{\prime} \varDelta_{L,\mathbf{z}}(t) \equiv 1+\sum_{i=1}^{r} z_i(1-\sum_{j \neq i} \ell_{ij}) \pmod{2} & \hbox{if $p=2$} 
\end{array} 
\]
for any $\mathbf{z}$ such that $\varDelta_{L,\mathbf{z}}(t) \neq 0$. 

{\em (iii)} For each $r \ge 2$ and for any $0 \le \ell, m \in \mathbb Z$, there is an $r$-component link $L$ such that 
\[
\lambda_{L,\mathbf{1}} = r-1+2\ell , \hspace*{10pt} \mu_{L,\mathbf{1}}=m 
\]
and $|H_1(M_{\mathbf{1},p^n};\mathbb Z)| \neq 0$ for any $n$. 
\end{theorem}

\begin{proof}
(i) Note that $\mathbf{z}=\pm 1$. Since $\varDelta_{L,\mathbf{z}}(1)=\pm 1$ and $\varPhi_{p^n}(1)=p$, $\varPhi_{p^n}(t)$ does not divide $\varDelta_{L,\mathbf{z}}(t)$ and $\varDelta_{L,\mathbf{z}}(1+T) \in \widehat{\varLambda}_1^{\times}$. Therefore $\lambda_{L,\mathbf{z}}=\mu_{L,\mathbf{z}}=0$ and $|H_1(M_{\mathbf{z},p^n};\mathbb Z)| \neq 0$ for any $n$ by Theorem \ref{Iwasawa type formula} (1). 

(ii) By Corollary 4.13.2 of \cite{Hil}, $\varDelta_{L,\mathbf{z}}(t)$ is divisible by $(t-1)^{r-1}$, i.e., $P_{L,\mathbf{z}}(T)$ is divisible by $T^{r-1}$. Therefore $\lambda_{L,\mathbf{z}} \ge r-1$. We may assume that $\varDelta_{L,\mathbf{z}}(t) \in \mathbb Z[t]$ and $\varDelta_{L,\mathbf{z}}(0) \neq 0$. Put $\lambda=\lambda_{L,\mathbf{z}}$, $\mu=\mu_{L,\mathbf{z}}$, $P(T)=P_{L,\mathbf{z}}(T)$ and $\dot{T}=(1+T)^{-1}-1$. By the Torres conditions (cf.\ \cite{Hil} etc.), we have $\varDelta_{L,\mathbf{z}}(t)=(-1)^{r-1} t^d \varDelta_{L,\mathbf{z}}(t^{-1})$ where $d=\deg \varDelta_{L,\mathbf{z}}(t) \equiv 1+\sum_{i=1}^{r} z_i(1-\sum_{j \neq i} \ell_{ij}) \pmod{2}$. Then we have 
\[
\varDelta_{L,\mathbf{z}}(1+T)=p^\mu P(T)U(T) = p^\mu (-1)^{r-1} (1+T)^d P(\dot{T})U(\dot{T}) 
\]
with a distinguished polynomial $P(T)=T^{\lambda}+\sum_{i=0}^{\lambda-1} c_i T^i\equiv T^{\lambda} \pmod{p}$ and $U(T) \in \widehat{\varLambda}_1^{\times} \cap \mathbb Z_p[T]$. 
Since 
\begin{eqnarray*}
(1+T)^{\lambda}P(\dot{T}) &=& (-T)^{\lambda}+\sum_{i=0}^{\lambda-1} c_i (-T)^i (1+T)^{\lambda-i} \\
&=& P(-1)T^{\lambda} + \sum_{j=0}^{\lambda-1} \sum_{i=0}^{j} c_i(-1)^i {\lambda-i \choose j-i} T^j 
\end{eqnarray*}
and the distinguished polynomial $P(T)$ dividing $\varDelta_{L,\mathbf{z}}(1+T)$ in $\widehat{\varLambda}_1$ is unique, then $P(T)=P(-1)^{-1}(1+T)^{\lambda}P(\dot{T})$ and hence $U(T)=(-1)^{r-1} P(-1) (1+T)^{d-\lambda} U(\dot{T})$. Since $U(0)=(-1)^{r-1} P(-1) U(0)$, we have $(-1)^{r-1}=P(-1) \equiv (-1)^{\lambda} \pmod{p}$ and $U(T)=(1+T)^{d-\lambda} U(\dot{T})$. Therefore $\lambda \equiv r-1 \pmod{2}$ if $p \neq 2$. Suppose that $p=2$, and put $u(t)=U(t-1) \in \mathbb Z_2[t]$. Since $u(t)=t^{d-\lambda}u(t^{-1})$ and $u(1) = U(0) \not\equiv 0 \pmod{2}$, then $\deg u(t)=d-\lambda$ must be even. Therefore $\lambda \equiv d \pmod{2}$ if $p=2$. 

(iii) Put $\nabla(t)=p^m t^{-\ell} (t-1)^{2\ell} \in \varLambda_1$. Since $\nabla(t)=\nabla(t^{-1})$ and $\deg^{\prime} \nabla(t)$ is even, there is an $r$-component link $L$ such that $\varDelta_{L,\mathbf{1}}(t)=(t-1)^{r-1} \nabla(t)$ by \cite{Hos58}. Then $\lambda_{L,\mathbf{1}} = r-1+2\ell$ and $\mu_{L,\mathbf{1}}=m$. By Theorem \ref{Iwasawa type formula} (1), we have $|H_1(M_{\mathbf{1},p^n};\mathbb Z)| \neq 0$ for all $n$. 
\end{proof}

\begin{remark}{\em 
We note that $\lambda_{L,\mathbf{1}} \equiv r-1 \pmod{2}$ even if $p=2$. If $p$ splits completely in $k$, we have an analogy: 
\[
T^{r-1} \big| P_{L,\mathbf{1}}(T),\ \lambda_{L,\mathbf{1}} \ge r-1 \hspace*{10pt}\longleftrightarrow\hspace*{10pt} T^{r_2} \big| P_{k_{\infty}^{\mathrm{cyc}}}(T),\ \lambda_{k_{\infty}^{\mathrm{cyc}}} \ge r_2
\]
where $k_{\infty}^{\mathrm{cyc}}/k$ is the cyclotomic $\mathbb Z_p$-extension (cf.\ e.g.\ \cite{Gre76}). Refer to \cite{FOO06} (resp.\ \cite{Oza04}) for the existence of $k_{\infty}/k$ with prescribed $\lambda_{k_{\infty}}$ (resp.\ $\mu_{k_{\infty}}$). 
}\end{remark}

\section{2-component links}

In this section, we consider the Iwasawa invariants of $2$-component links. 

\begin{lemma}\label{r2lemma}
For a $2$-component link $L=K_1 \cup K_2$ and any $\mathbf{z}=(z_1,z_2)$ with $\gcd(z_1,z_2)=1$, we have $|H_1(M_{\mathbf{z},p^n};\mathbb Z)| \neq 0$ for all $0 \le n \le v=\max{\{v_p(z_1),v_p(z_2)\}}=v_p(z_1z_2)$. 
\end{lemma}

\begin{proof}
We may assume that $v=v_p(z_2) \ge v_p(z_1) =0$. Since the covering space $X_{\mathbf{z},p^n}$ for $n \le v$ corresponds to the kernel of 
\[
z_1^{-1} \circ \bmod{p^n} \circ \sigma_{\mathbf{z}} : H_1(X;\mathbb Z) \rightarrow \mathbb Z/p^n \mathbb Z : t_1 \mapsto 1 \bmod{p^n}, t_2 \mapsto 0 \bmod{p^n} , 
\]
then $M_{\mathbf{z},p^n}$ is the cyclic branched covering space along $K_1$. Since $\varDelta_{K_1}(1) = \pm 1 \not\equiv 0 \pmod{p}$, then $\varPhi_{p^n}(t)$ does not divide $\varDelta_{K_1}(t)$ for all $n \ge 1$, and hence $|H_1(M_{\mathbf{z},p^n};\mathbb Z)|= \pm \prod_{\zeta^{p^n}=1} \varDelta_{K_1}(\zeta) \neq 0$ for any $n \le v$. 
\end{proof}

The following theorem shows the relationship between the linking numbers and Iwasawa invariants. 

\begin{theorem}\label{lambda1mu0}
For a $2$-component link $L$ and any $\mathbf{z}$, $\lambda_{L,\mathbf{z}}=1$, $\mu_{L,\mathbf{z}}=0$ if and only if $v_p(\ell_{12})=0$. Moreover, then $|H_1(M_{\mathbf{z},p^n};\mathbb Z)| \neq 0$ for any $n$. 
\end{theorem}

\begin{proof}
By the Torres conditions \cite{Tor53}, we have 
\[
\varDelta_L(t_1,t_2)=\frac{(t_1t_2)^{\ell_{12}}-1}{t_1t_2-1}\varDelta_{K_1}(t_1)\varDelta_{K_2}(t_2)+(t_1-1)(t_2-1)f(t_1,t_2) 
\]
with some $f(t_1,t_2) \in \varLambda_2$. Then $\varDelta_{L,\mathbf{z}}(1+T)/T \equiv \varDelta_L(1,1)=\pm \ell_{12} \pmod{T}$. Suppose $v_p(\ell_{12})=0$. Then $\varPhi_{p^n}(t)$ does not divide $\varDelta_{L,\mathbf{z}}(t)/(t-1)$ for any $n \ge 1$ and $\varDelta_{L,\mathbf{z}}(1+T)/T \in \widehat{\varLambda}_1^{\times}$. By Lemma \ref{r2lemma} and Theorem \ref{Iwasawa type formula} (1), we have $\lambda_{L,\mathbf{z}}=1$, $\mu_{L,\mathbf{z}}=0$ and $|H_1(M_{\mathbf{z},p^n};\mathbb Z)| \neq 0$ for any $n$. If $v_p(\ell_{12})>0$, then $\varDelta_{L,\mathbf{z}}(1+T)/T \not\in \widehat{\varLambda}_1^{\times}$ and hence $\lambda_{L,\mathbf{z}} \ge 2$ or $\mu_{L,\mathbf{z}} \ge 1$. 
\end{proof}

Since the linking numbers are analogous to the power residue symbols, Theorem \ref{lambda1mu0} can be regarded as an analogy of Gold-Sands's theorem (cf.\ \cite{Gol74}, \cite{San} Proposition 2.1). Let $k$ be an imaginary quadratic field in which $p$ splits completely. Then $S=\{\wp_1,\wp_2\}$ and $r_2=1$. There is some $\pi_2 \in k$ such that $\wp_2^{h_k}=(\pi_2)$ as a principal ideal. Gold-Sands's theorem implies that $\lambda_{k_{\infty}}=1$ and $\mu_{k_{\infty}}=0$ (i.e., $Y_{k_{\infty}}=\mathrm{Gal}(\widetilde{k}/k_{\infty}) \simeq \mathbb Z_p$) for any $\mathbb Z_p$-extensions $k_{\infty}/k$ in which both $\wp_1$ and $\wp_2$ ramify if and only if $h_k \not\equiv 0 \pmod{p}$ and $\pi_2^{p-1} \not\equiv 1 \pmod{\wp_1^2}$, i.e., $\pi_2$ is not $p$th power residue modulo $\wp_1^2$. 
\[
\begin{array}{ccc}
L=K_1 \cup K_2,\ G_L/G_L^{\prime} \simeq \mathbb Z^2 &\longleftrightarrow& S=\{\wp_1,\wp_2\},\ G_S/G_S^{\prime} \simeq \mathbb Z_p^2 \\
|H_1(S^3;\mathbb Z)|=1,\ v_p(\ell_{12})=0 &\longleftrightarrow& h_k \not\equiv 0 \!\!\!\pmod{p},\ \pi_2^{p-1} \not\equiv 1 \!\!\!\pmod{\wp_1^2} \\
\Updownarrow && \Updownarrow \\
\lambda_{L,\mathbf{z}}=1,\ \mu_{L,\mathbf{z}}=0 &\longleftrightarrow& \lambda_{k_{\infty}}=1,\ \mu_{k_{\infty}}=0 \\
\end{array}
\]

As a generalization of this result, Ozaki \cite{Oza01} proved that if Greenberg's conjecture (cf.\ \S4) holds for $\widetilde{k}/k$ then $\lambda_{k_{\infty}}=1$ and $\mu_{k_{\infty}}=0$ for all but finitely many $k_{\infty}/k$. On the other hand, the following theorem implies that if $v_p(\ell_{12})>0$, in particular $\ell_{12}=0$, then $\lambda_{L,\mathbf{z}}$ is not necessary bounded as $\mathbf{z}$ is varied. 

\begin{figure}[h]
\caption{$L=K_1 \cup K_2$, $\ell_{12}=0$}\label{fig1}
\begin{center}
\unitlength 0.1in
\begin{picture}(48.25,15.00)(-0.25,-15.00)
%
\special{pn 8}%
\special{pa 1000 1100}%
\special{pa 1400 700}%
\special{fp}%
\special{pa 1000 700}%
\special{pa 1150 850}%
\special{fp}%
\special{pa 1400 1100}%
\special{pa 1250 950}%
\special{fp}%
%
\special{pn 8}%
\special{pa 1400 1100}%
\special{pa 1800 700}%
\special{fp}%
\special{pa 1400 700}%
\special{pa 1550 850}%
\special{fp}%
\special{pa 1800 1100}%
\special{pa 1650 950}%
\special{fp}%
%
\special{pn 8}%
\special{pa 1800 1100}%
\special{pa 2200 700}%
\special{fp}%
\special{pa 1800 700}%
\special{pa 1950 850}%
\special{fp}%
\special{pa 2200 1100}%
\special{pa 2050 950}%
\special{fp}%
%
\special{pn 8}%
\special{pa 2200 1100}%
\special{pa 2600 700}%
\special{fp}%
\special{pa 2200 700}%
\special{pa 2350 850}%
\special{fp}%
\special{pa 2600 1100}%
\special{pa 2450 950}%
\special{fp}%
%
\special{pn 8}%
\special{pa 3000 1100}%
\special{pa 3400 700}%
\special{fp}%
\special{pa 3000 700}%
\special{pa 3150 850}%
\special{fp}%
\special{pa 3400 1100}%
\special{pa 3250 950}%
\special{fp}%
%
\special{pn 8}%
\special{pa 3400 1100}%
\special{pa 3800 700}%
\special{fp}%
\special{pa 3400 700}%
\special{pa 3550 850}%
\special{fp}%
\special{pa 3800 1100}%
\special{pa 3650 950}%
\special{fp}%
%
\special{pn 8}%
\special{ar 4200 900 200 200  2.3561945 4.7123890}%
%
\special{pn 8}%
\special{ar 4000 900 200 200  5.4977871 6.2831853}%
\special{ar 4000 900 200 200  0.0000000 1.5707963}%
%
\special{pn 8}%
\special{pa 3800 1100}%
\special{pa 4000 1100}%
\special{fp}%
\special{pa 4200 700}%
\special{pa 4400 700}%
\special{fp}%
\special{pa 3800 700}%
\special{pa 4050 700}%
\special{fp}%
\special{pa 4400 1100}%
\special{pa 4150 1100}%
\special{fp}%
%
\special{pn 8}%
\special{pa 800 700}%
\special{pa 1000 700}%
\special{fp}%
\special{pa 800 1100}%
\special{pa 1000 1100}%
\special{fp}%
%
\special{pn 8}%
\special{ar 2800 300 200 300  1.1071487 5.1760366}%
%
\special{pn 8}%
\special{ar 2800 300 200 300  5.8644972 6.2831853}%
\special{ar 2800 300 200 300  0.0000000 0.4186882}%
%
\special{pn 8}%
\special{pa 800 500}%
\special{pa 2600 500}%
\special{fp}%
\special{pa 4400 500}%
\special{pa 2750 500}%
\special{fp}%
\special{pa 800 100}%
\special{pa 2600 100}%
\special{fp}%
\special{pa 4400 100}%
\special{pa 2750 100}%
\special{fp}%
%
\special{pn 8}%
\special{ar 4000 1400 200 100  4.7123890 6.2831853}%
\special{ar 4000 1400 200 100  0.0000000 0.6435011}%
%
\special{pn 8}%
\special{ar 4200 1400 200 100  1.5707963 3.7850938}%
%
\special{pn 8}%
\special{pa 4400 1500}%
\special{pa 4200 1500}%
\special{fp}%
\special{pa 4400 1300}%
\special{pa 4200 1300}%
\special{fp}%
%
\special{pn 8}%
\special{pa 800 1300}%
\special{pa 4000 1300}%
\special{fp}%
\special{pa 800 1500}%
\special{pa 4000 1500}%
\special{fp}%
%
\special{pn 8}%
\special{ar 800 1200 100 100  1.5707963 4.7123890}%
%
\special{pn 8}%
\special{ar 4400 1200 100 100  4.7123890 6.2831853}%
\special{ar 4400 1200 100 100  0.0000000 1.5707963}%
%
\special{pn 8}%
\special{ar 800 600 100 100  1.5707963 4.7123890}%
%
\special{pn 8}%
\special{ar 4400 600 100 100  4.7123890 6.2831853}%
\special{ar 4400 600 100 100  0.0000000 1.5707963}%
%
\special{pn 8}%
\special{ar 800 800 400 700  1.5707963 4.7123890}%
%
\special{pn 8}%
\special{ar 4400 800 400 700  4.7123890 6.2831853}%
\special{ar 4400 800 400 700  0.0000000 1.5707963}%
\put(24.0000,-3.0000){\makebox(0,0){$K_1$}}%
\put(2.0000,-8.0000){\makebox(0,0){$K_2$}}%
%
\special{pn 8}%
\special{pa 2600 700}%
\special{pa 3000 700}%
\special{dt 0.045}%
\special{pa 3000 700}%
\special{pa 2999 700}%
\special{dt 0.045}%
\special{pa 2600 1100}%
\special{pa 3000 1100}%
\special{dt 0.045}%
\special{pa 3000 1100}%
\special{pa 2999 1100}%
\special{dt 0.045}%
\put(14.0000,-9.0000){\makebox(0,0){$1$}}%
\put(22.0000,-9.0000){\makebox(0,0){$2$}}%
\put(34.0000,-9.0000){\makebox(0,0){$m$}}%
\end{picture}%
\end{center}
\end{figure}
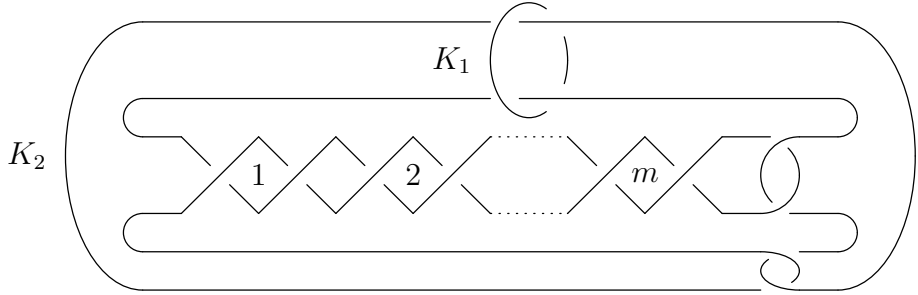

\begin{theorem}\label{biglambda}
{\em (1)} Let $L=K_1 \cup K_2$ be a $2$-component link as in Figure \ref{fig1}, where $1 \le m \in \mathbb Z$. Then 
\[
\lambda_{L,\mathbf{z}}=1+p^{v_p(z_1)}+3p^{v_p(z_2)}, \hspace*{10pt} \mu_{L,\mathbf{z}}=v_p(m)
\]
and $|H_1(M_{\mathbf{z},p^n};\mathbb Z)| \neq 0$ for any $n$. 

{\em (2)} Let $L=C(2a,2b,-2a)$ be a $2$-bridge link in a Conway's notation, where $0 \neq a, b \in \mathbb Z$. Then 
\[
\lambda_{L,\mathbf{z}}=2+p^{v_p(z_1z_2)}+(p^{v_p(a)}-1)p^{v_p(z_1+z_2)} , \hspace*{10pt} \mu_{L,\mathbf{z}}=v_p(b) . 
\]
Moreover, $|H_1(M_{\mathbf{z},p^n};\mathbb Z)| \neq 0$ for any $n$ if and only if $v_p(z_1z_2) \ge v_p(a)$. 
\end{theorem}

\begin{proof}
For any $1 \le z \in \mathbb Z$, we note that 
\[
\frac{t^{z}-1}{t^{p^{v_p(z)}}-1} = \sum_{i=0}^{z/p^{v_p(z)}-1} (t^{p^{v_p(z)}})^i \equiv z/p^{v_p(z)} \in \mathbb Z_p^{\times} \pmod{t-1} . 
\]

(1) By direct computations, one can see that $\varDelta_L(t_1,t_2)=m(t_1-1)(t_2-1)^3$. Then $\varDelta_{L,\mathbf{z}}(t)=m(t-1)(t^{z_1}-1)(t^{z_2}-1)^3$ which is not divisible by $\varPhi_{p^n}(t)$ for all $n > v$. Therefore $|H_1(M_{\mathbf{z},p^n};\mathbb Z)| \neq 0$ for any $n$ by Theorem \ref{Iwasawa type formula} (1) and Lemma \ref{r2lemma}. Since
\[
\varDelta_{L,\mathbf{z}}(1+T)=p^{v_p(m)}T ((1+T)^{p^{v_p(z_1)}}-1) ((1+T)^{p^{v_p(z_2)}}-1)^3 U(T)
\]
with some $U(T) \in \widehat{\varLambda}_1^{\times}$, we obtain the claim (1). 

(2) By \cite{Kan84} etc., we have 
\[
\varDelta_L(t_1,t_2)=b(t_1-1)(t_2-1)\frac{(t_1t_2)^a-1}{t_1t_2-1} . 
\]
Put $v=v_p(z_1z_2)$. Then $\varDelta_{L,\mathbf{z}}(t)$ is divisible by $\varPhi_{p^n}(t)$ if and only if $n \le v$ or $v_p(z_1+z_2) < n \le v_p(a(z_1+z_2))$. Note that $v_p(z_1+z_2)=0$ if $v \neq 0$. By Theorem \ref{Iwasawa type formula} (1) and Lemma \ref{r2lemma}, $|H_1(M_{\mathbf{z},p^n};\mathbb Z)|=0$ for some $n>v$ if and only if $v_p(a) > 0$ and $v < v_p(a(z_1+z_2))$, i.e., $v_p(a)>v$. Since 
\[
\varDelta_{L,\mathbf{z}}(1+T)= p^{v_p(b)}T^2((1+T)^{p^{v}}-1)\frac{(1+T)^{p^{v_p(a(z_1+z_2))}}-1}{(1+T)^{p^{v_p(z_1+z_2)}}-1}
U(T)
\]
with some $U(T) \in \widehat{\varLambda}_1^{\times}$, we obtain the claim (2). 
\end{proof}

In the case where $p=2$, Theorem \ref{mainthm} (iii) and the following theorem show the existence of a pair of $2$-component link $L$ and $\mathbf{z}$ with prescribed $\lambda_{L,\mathbf{z}} \ge 1$ and $\mu_{L,\mathbf{z}} \ge 0$. 

\begin{theorem}
Assume that $p=2$. For any $0 \le \ell, m \in \mathbb Z$, there is a pair of $2$-component link $L$ and $\mathbf{z}$ such that 
\[
\lambda_{L,\mathbf{z}} = 2+2\ell , \hspace*{10pt} \mu_{L,\mathbf{z}}=m 
\]
and $|H_1(M_{\mathbf{z},2^n};\mathbb Z)| \neq 0$ for any $n$. 
\end{theorem}

\begin{proof}
First, we consider the case $\ell \ge 1$. Put 
\[
\varDelta(t_1,t_2)=2^m (t_1-1)(t_2-1)(t_1t_2+t_1^{-1}t_2^{-1})^{\max\{0,\ell-2\}}
\]
and put $\mathbf{z}=(1,4)$ if $\ell \ge 2$, and $\mathbf{z}=(1,2)$ if $\ell=1$. By a consequence of Bailey's theorem (cf.\ \cite{Hil} Corollary 7.4.1, \cite{Lev88}), there is a $2$-component link $L$ such that $\varDelta_L=\varDelta(t_1,t_2)$ (and $\ell_{12}=0$). Then 
\[
\varDelta_{L,\mathbf{z}}(t)=
\bigg\{
\begin{array}{ll}
2^m (t-1)^2(t^4-1)(t^{10}+1)^{\ell-2} t^{-5(\ell-2)} & \hbox{if $\ell \ge 2$} \\ 
2^m (t-1)^2(t^2-1) & \hbox{if $\ell =1$} 
\end{array} 
\]
and $|H_1(M_{\mathbf{z},2^n};\mathbb Z)| \neq 0$ for any $n$ by Theorem \ref{Iwasawa type formula} (1) and Lemma \ref{r2lemma}. Since 
\[
\varDelta_{L,\mathbf{z}}(1+T)=
\bigg\{
\begin{array}{ll}
2^m T^2((1+T)^4-1)((1+T)^2+1)^{\ell-2} U(T) & \hbox{if $\ell \ge 2$} \\ 
2^m T^2((1+T)^2-1) U(T) & \hbox{if $\ell =1$} 
\end{array}
\]
with some $U(T) \in \widehat{\varLambda}_1^{\times}$, we have $\lambda_{L,\mathbf{z}} = 2+2\ell$ and $\mu_{L,\mathbf{z}}=m$. 

If $\ell=m=0$, put $L=C(4)$ the $2$-bridge link, and put $\mathbf{z}=(1,2)$. Then we have $\varDelta_L=t_1t_2+1$ by \cite{Kan84} and hence $\varDelta_{L,\mathbf{z}}(t)=(t-1)(t^3+1)$. By Theorem \ref{Iwasawa type formula} (1) and Lemma \ref{r2lemma}, $|H_1(M_{\mathbf{z},2^n};\mathbb Z)| \neq 0$ for any $n$. Since $\varDelta_{L,\mathbf{z}}(1+T)=T((1+T)^3+1) \equiv T^2(1+T+T^2) \not\equiv 0 \pmod{2}$, we have $\lambda_{L,\mathbf{z}}=2$ and $\mu_{L,\mathbf{z}}=0$. 

Suppose that $\ell=0$ and $m \ge 1$. Put $\nu(t)=(t^{2^m}-t^{-2^m})/(t-t^{-1})$ and $\beta(t)=t^{-2^m}(t-1)(t^{2^{m+1}-1}-1) \in \varLambda_1$. Since $\nu(t^{-1})=\nu(t)$ and $\beta(t^{-1})=\beta(t)$, there are monic polynomials $N(x)$, $B(x) \in \mathbb Z[x]$ such that $N(t+t^{-1})=\nu(t)$ and $B(t+t^{-1})=\beta(t)$. Then $\deg N(x)=2^m-1$ and $\deg B(x)=2^m$. Note that the determinant of the Sylvester matrix $Q$ of $N(x)$ and $B(x)$ is the resultant $\mathrm{Res}_{\mathbb Z}(N,B)$ of $N(x)$ and $B(x)$. Let $\zeta$ be a primitive $2^{m+1}$th root of $1$. Since the roots of $N(x)$ are $\zeta^i+\zeta^{-i}$ $(i=1,2,\cdots,2^m-1)$, we have 
\[
|\mathrm{Res}_{\mathbb Z}(N,B)|^2
= \left|\prod_{i=1}^{2^m-1} B(\zeta^i+\zeta^{-i})\right|^2
= \left|\prod_{i=1}^{2^{m+1}-1} \beta(\zeta^i)\right| |\beta(-1)|^{-1}
= 2^{2m}
\]
and hence $\mathrm{Res}_{\mathbb Z}(N,B)=\pm 2^m$. Since each row of $Q$ is contained in the kernel of the surjective homomorphism $\mathbb Z^{2^{m+1}-1} \rightarrow \mathbb Z[x]/(N,B) : (a_{2^{m+1}-2},\cdots,a_1,a_0) \mapsto a_{2^{m+1}-2}x^{2^{m+1}-2}+\cdots+a_1x+a_0 \bmod{(N,B)}$, we have 
\[
2^m = \pm \mathrm{Res}_{\mathbb Z}(N,B) \in E_{\mathbb Z}^{(0)}(\mathbb Z[x]/(N,B)) \subset \mathrm{Ann}_{\mathbb Z}(\mathbb Z[x]/(N,B)) . 
\]
Therefore there exist some $F(x)$, $G(x) \in \mathbb Z[x]$ such that $2^m=N(x)F(x)+B(x)G(x)=\nu(t)F(t+t^{-1})+\beta(t)G(t+t^{-1})$. Put $f(t_1,t_2)=F(t_1t_2+t_1^{-1}t_2^{-1})$, $g(t_1,t_2)=G(t_1t_2+t_1^{-1}t_2^{-1})$ and put 
\[
\varDelta(t_1,t_2)=\frac{(t_1t_2)^{2^{m+1}}-1}{t_1t_2-1} f(t_1,t_2) - (t_1-1)(t_2-1) \frac{(t_1t_2)^{2^{m+1}-1}-1}{t_1t_2-1} g(t_1,t_2). 
\]
Since $\nu(1)=2^m$ and $\beta(1)=0$, then $f(1,1)=F(2)=1$ and hence $F(t+t^{-1})$ is not divisible by $\varPhi_{2^n}(t)$ for all $n$. By a consequence of Bailey's theorem (cf.\ \cite{Hil} Theorem 7.4, \cite{Lev88}), there is a $2$-component link $L$ such that $\varDelta_L=f(t_1,t_2)^z \varDelta(t_1,t_2)$ for some $z \ge 0$ (and $\ell_{12}=2^{m+1}$). For $\mathbf{z}=(2,-1)$, we have 
\begin{eqnarray*}
\varDelta_{L,\mathbf{z}}(t) &=& (t^2-1)F(t+t^{-1})^z \big( \nu(t)F(t+t^{-1})+\beta(t)G(t+t^{-1}) \big) \\
&=& 2^m (t^2-1)F(t+t^{-1})^z 
\end{eqnarray*}
and $|H_1(M_{\mathbf{z},2^n};\mathbb Z)| \neq 0$ for any $n$ by Theorem \ref{Iwasawa type formula} (1) and Lemma \ref{r2lemma}. Since $F((1+T)+(1+T)^{-1}) \in \widehat{\varLambda}_1^{\times}$, we have $\lambda_{L,\mathbf{z}} = 2$ and $\mu_{L,\mathbf{z}}=m$. 
\end{proof}

\begin{remark}{\em 
If $\ell_{12}=0$, then $\varDelta_L(t_1,t_2)$ is divisible by $(t_1-1)(t_2-1)$ and therefore $\lambda_{L,\mathbf{z}} \ge 3$. Since $\varDelta_L(1,1)=\pm \ell_{12}$, we have $\mu_{L,\mathbf{z}} \le v_p(\ell_{12})$ if $\ell_{12} \neq 0$. If $p=2$ and $\lambda_{L,\mathbf{z}}$ is even, then $\ell_{12}$ is even and $z_1+z_2$ is odd by Theorem \ref{mainthm} (ii). 
}\end{remark}

\section{Greenberg type problem}

We shall consider a problem analogous to Greenberg's conjecture. Put the ring $\widehat{\varLambda}_r=\mathbb Z_p[[T_1,T_2,\cdots,T_r]]$ of $r$-variable formal power series. Let $G$ be a pro-$p$ group with a surjective homomorphism $\widehat{\psi} : G \rightarrow G/G^{\prime}=\prod_{i=1}^{r} \gamma_i^{\mathbb Z_p} \simeq \mathbb Z_p^r$, and $G^{\prime\prime}$ be the commutator subgroup of $G^{\prime}$. Then $G^{\prime}/G^{\prime\prime}$ is a module over $\widehat{\varLambda}_r \simeq \mathbb Z_p[[G/G^{\prime}]] : 1+T_i \leftrightarrow \gamma_i$, and the completed differential module of $G$ is defined as 
\[
\mathfrak{A}_{G}=\bigoplus_{g \in G} \widehat{\varLambda}_r dg \Big/ \big\langle d(g_1g_2)-dg_1-\widehat{\psi}(g_1)dg_2\,(g_i \in G) \big\rangle_{\widehat{\varLambda}_r}
\]
(cf.\ \cite{Mor} \S9.3). For the pro-$p$ group $G_S$, $\mathfrak{X}_S=G_S^{\prime}/G_S^{\prime\prime}$ is the $S$-ramified Iwasawa module which is a finitely generated $\widehat{\varLambda}_{r_2+1}$-module, and we have the completed Crowell exact sequence 
\[
0 \rightarrow \mathfrak{X}_S \stackrel{\widehat{\theta}_1}{\rightarrow} \mathfrak{A}_{G_S} \stackrel{\widehat{\theta}_2}{\rightarrow} \widehat{\varLambda}_{r_2+1} \rightarrow \mathbb Z_p \rightarrow 0
\]
where $\widehat{\theta}_1(gG_S^{\prime\prime})=dg$ and $\widehat{\theta}_2(dg)=\widehat{\psi}(g)-1$ (cf.\ \cite{Mor02} \S2.2, \cite{Mor} \S9.4, \cite{NQD83}). Let $M_S$ be the $\widehat{\varLambda}_{r_2+1}$-submodule of $\mathfrak{A}_{G_S}$ generated by $d\tau$ where $\tau$ are the elements of the inertia subgroups of $G_S$. Then $M_S \cap \mathrm{Ker}\,\widehat{\theta}_2$ is generated by $d\tau$ such that $\tau|_{\widetilde{k}}=1$, and hence $Y_{\widetilde{k}}=\mathfrak{X}_S/{\widehat{\theta}_1}^{-1}(M_S \cap \mathrm{Ker}\,\widehat{\theta}_2)$ is the Galois group of the maximal unramified abelian pro-$p$-extension of $\widetilde{k}$. Greenberg \cite{Gre01} conjectured that the unramified Iwasawa module $Y_{\widetilde{k}}$ is a pseudonull $\widehat{\varLambda}_{r_2+1}$-module. 

The analogue of $Y_{\widetilde{k}}$ is defined in a similar way as follows. Let $\psi : G_L \rightarrow G_L/G_L^{\prime}$ be the natural surjective homomorphism, and put the differential module 
\[
\mathfrak{A}_{L}=\bigoplus_{g \in G_L} \varLambda_r dg \Big/ \big\langle d(g_1g_2)-dg_1-\psi(g_1)dg_2\,(g_i \in G_L) \big\rangle_{\varLambda_r}
\]
of $G_L$. Let $[\widetilde{m}_i] \in A_L$ be the class of a lift $\widetilde{m}_i$ of $m_i$ with endpoints in $\pi^{-1}(\ast)$, and $M_L$ the $\varLambda_r$-submodule of $A_L$ generated by $[\widetilde{m}_i]$ ($1 \le i \le r$). Since $\mathfrak{A}_{L} \simeq A_L : d[m_i] \leftrightarrow [\widetilde{m}_i]$, $M_L$ is an analogue of $M_S$, and we have the Crowell exact sequence 
\[
0 \rightarrow B_L \stackrel{\theta_1}{\rightarrow} A_L \stackrel{\theta_2}{\rightarrow} \varLambda_r \rightarrow \mathbb Z \rightarrow 0
\]
(cf.\ \cite{Hil} p.70, 78). Thus the analogous module $Y_L=B_L/{\theta_1}^{-1}(M_L \cap \mathrm{Ker}\,\theta_2)$ is defined. 
\[
\hbox{``unramified" link module }Y_L \hspace*{10pt}\longleftrightarrow\hspace*{10pt} \hbox{unramified Iwasawa module }Y_{\widetilde{k}}
\]

\begin{remark}{\em 
It is known that $\mathrm{rank}_{\widehat{\varLambda}_{r_2+1}} \mathfrak{X}_S=r_2$ and $\mathfrak{X}_S$ has no nonzero pseudonull $\widehat{\varLambda}_{r_2+1}$-submodule (cf.\ \cite{Gre78} \cite{NQD83}). As analogous properties, we know that $\mathrm{rank}_{\varLambda_r} B_L \le r-1$ and that $B_L$ has no nonzero pseudonull $\varLambda_r$-submodule if $\mathrm{rank}_{\varLambda_r} B_L=0$ (cf.\ \cite{Kaw96} Corollary 7.3.13, \cite{Hil} Theorem 4.12 (2)). 
\[
\mathrm{rank}_{\varLambda_r} B_L \le r-1 \hspace*{10pt}\longleftrightarrow\hspace*{10pt} \mathrm{rank}_{\widehat{\varLambda}_{r_2+1}} \mathfrak{X}_S=r_2
\]
}\end{remark}

Then the following problem arises as an analogue of Greenberg's conjecture. 

\begin{problem}
Is $Y_L$ a pseudonull $\varLambda_r$-module ?
\end{problem}

For this problem, we obtain the following criteria and examples. We denote by $\varLambda_{r,J}$ the localization of $\varLambda_r$ by the multiplicative system $J=\{f \in \varLambda_r \,|\, f(1,1,\cdots,1)=1\}$. Note that $\varLambda_{r,J}$ is flat over $\varLambda_r$ (cf.\ \cite{Mat} \S7).

\begin{theorem}
{\em (i)} If $\varDelta_L$ has no prime factor $f$ such that $f(1,1,\cdots,1)=\pm 1$, then $Y_L$ is a pseudonull $\varLambda_r$-module. 

{\em (ii)} If $r=1$, then $Y_L$ is not a pseudonull $\varLambda_1$-module. 

\end{theorem}

\begin{proof}
(i) By Levine's theorem \cite{Lev83} (cf.\ \cite{Hil} Theorem 4.13), we have $(A_L/M_L) \otimes_{\varLambda_r}\varLambda_{r,J}=0$ and hence $Y_L \otimes_{\varLambda_r}\varLambda_{r,J}=0$. Then there exists some $f \in \mathrm{Ann}_{\varLambda_r}Y_L$ such that $f \in J$. On the other hand, $\varDelta_L \in E_{\varLambda_r}^{(0)}(B_L) \subset \mathrm{Ann}_{\varLambda_r}B_L \subset \mathrm{Ann}_{\varLambda_r}Y_L$. Since $\varDelta_L$ and $f$ are relatively prime in $\varLambda_r$ by the assumption, the divisorial hull of $\mathrm{Ann}_{\varLambda_r}Y_L$ is $\varLambda_r$, i.e., $Y_L$ is pseudonull. 

(ii) By Theorem 4.12 (1) (2) of \cite{Hil}, $A_L \simeq B_L \oplus \varLambda_1$, and $B_L$ is $\varLambda_1$-torsion and has no nonzero pseudonull $\varLambda_1$-submodule. Since $(A_L/M_L) \otimes_{\varLambda_1}\varLambda_{1,J}=0$ by Levine's theorem \cite{Lev83} again, $M_L=\varLambda_1 [\widetilde{m}_1] \simeq \varLambda_1$ and hence $M_L \cap \mathrm{Ker}\,\theta_2=M_L \cap \theta_1(B_L)=0$. Therefore $Y_L=B_L$ is not pseudonull. 
\end{proof}

\begin{example}{\em 
(1) Let $L$ be a $2$-component link as in Figure \ref{fig1}. Then $\varDelta_L=m(t_1-1)(t_2-1)^3$ and hence $Y_L$ is a pseudonull $\varLambda_2$-module. 

(2) Let $L=C(2p^m,2b,-2p^m)$ be a $2$-bridge link where $p$ is a prime number. Then we have $\varDelta_L=b(t_1-1)(t_2-1) \prod_{n=1}^{m} \varPhi_{p^n}(t_1t_2)$ as in the proof of Theorem \ref{biglambda} (2). Note that $\varPhi_{p^n}(t_1t_2) \not\in J$ since $\varPhi_{p^n}(1)=p$. Suppose $\varPhi_{p^n}(t_1t_2)$ has a factor $f \in J \cap \mathbb Z[t_1,t_2]$ such that $f(0,0) \neq 0$, and put $g=\varPhi_{p^n}(t_1t_2)/f$. We denote by $\deg_{\,i}$ the degree as a polynomial of one variable $t_i$. Since $\varPhi_{p^n}(t_1)$ is irreducible, we have $g(t_1,1)=\varPhi_{p^n}(t_1)$ and $f(t_1,1)=1$. Then $\deg_1 g(t_1,t_2) \ge \deg_1 g(t_1,1) = \deg_1 \varPhi_{p^n}(t_1t_2)$ and hence $\deg_1 f(t_1,t_2)=0$. Similarly, $\deg_2 f(t_1,t_2)=0$. This implies that $f=1$. Therefore $Y_L$ is a pseudonull $\varLambda_2$-module. 
}\end{example}

\section{Pro-$p$ link module}

We can also regard the pro-$p$ completion $\widehat{G}_L$ of $G_L$ as an analogue of $G_S$. Then $\widehat{G}_L/\widehat{G}_L^{\prime} \simeq \mathbb Z_p^r$, where $\widehat{G}_L^{\prime}$ is the commutator subgroup of $\widehat{G}_L$. By identifying $t_i=T_i+1$, $\varLambda_r$ and $\varLambda_{r,J}$ are regarded as subrings of $\widehat{\varLambda}_r \simeq \mathbb Z_p[[\widehat{G}_L/\widehat{G}_L^{\prime}]]$. Since $\widehat{\varLambda}_r \simeq \varprojlim \varLambda_{r,J}/\mathfrak{m}^n \simeq \varprojlim \varLambda_r/\mathfrak{m}^n$ where $\mathfrak{m}=(p,T_1,T_2,\cdots,T_r)$ is the maximal ideal, $\widehat{\varLambda}_r$ is flat over $\varLambda_r$ and over $\varLambda_{r,J}$ (cf.\ \cite{Mat} Theorem 8.8). By the following proposition, $\widehat{B}_L=B_L \otimes_{\varLambda_r}\widehat{\varLambda}_r$ is also an analogue of $\mathfrak{X}_S$. Let $\widehat{G}_L^{\prime\prime}$ be the commutator subgroup of $\widehat{G}_L^{\prime}$, and put the completed Alexander module $\widehat{A}_L=A_L \otimes_{\varLambda_r}\widehat{\varLambda}_r$. 

\begin{proposition}
$\widehat{B}_L \simeq \widehat{G}_L^{\prime}/\widehat{G}_L^{\prime\prime}$. 
\end{proposition}

\begin{proof}
There is a natural homomorphism $\varphi : \widehat{B}_L \simeq (G_L^{\prime}/G_L^{\prime\prime})\otimes_{\varLambda_r}\widehat{\varLambda}_r \rightarrow \widehat{G}_L^{\prime}/\widehat{G}_L^{\prime\prime}$. For the natural homomorphism $\widehat{\psi} : \widehat{G}_L \rightarrow \widehat{G}_L/\widehat{G}_L^{\prime}$, the completed differential module $\mathfrak{A}_{\widehat{G}_L}$ of $\widehat{G}_L$ is defined. By \cite{HMM} (3.1.2) (3.1.6) and \cite{Mor} Theorem 9.14, we have $\mathfrak{A}_{\widehat{G}_L} \simeq \widehat{A}_L$. Then there is a commutative diagram 
\[
\begin{array}{rcccccccl}
0 \rightarrow & \widehat{B}_L & \stackrel{\theta_1 \otimes id}{\rightarrow} & \widehat{A}_L & \stackrel{\theta_2 \otimes id}{\rightarrow} & \varLambda_r \otimes_{\varLambda_r} \widehat{\varLambda}_r & \rightarrow & \mathbb Z \otimes_{\varLambda_r} \widehat{\varLambda}_r & \rightarrow 0 \\
& \downarrow^{\varphi} && \rotatebox[origin=c]{90}{$\simeq$} && \rotatebox[origin=c]{90}{$\simeq$} && \rotatebox[origin=c]{90}{$\simeq$} & \\
0 \rightarrow & \widehat{G}_L^{\prime}/\widehat{G}_L^{\prime\prime} & \rightarrow & \mathfrak{A}_{\widehat{G}_L} & \rightarrow & \widehat{\varLambda}_r & \rightarrow & \mathbb Z_p & \rightarrow 0 
\end{array}
\]
with exact rows, where the bottom row is the completed Crowell sequence for $\widehat{G}_L$. Therefore $\varphi$ is an isomorphism. 
\end{proof}

The following proposition is a background of Theorem \ref{lambda1mu0}. 

\begin{proposition}
If $r=2$, then $\widehat{B}_L \simeq \widehat{\varLambda}_2/(\varDelta_L(1+T_1,1+T_2))$. Moreover, if $v_p(\ell_{12})=0$, then $\widehat{B}_L=0$. 
\end{proposition}

\begin{proof}
Since $E_{\varLambda_2}^{(0)}(B_L)=\varLambda_2 \varDelta_L$ (cf.\ \cite{Hil} Theorem 4.6 (4)) and the presentation matrix of $B_L$ is the same to that of $\widehat{B}_L$, we have $E_{\widehat{\varLambda}_2}^{(0)}(\widehat{B}_L)=\widehat{\varLambda}_2 \varDelta_L(1+T_1,1+T_2)$. By Theorem 4.7 of \cite{Hil}, $\widehat{B}_L \simeq (B_L \otimes_{\varLambda_2} \varLambda_{2,J}) \otimes_{\varLambda_{2,J}} \widehat{\varLambda}_2$ is a cyclic $\widehat{\varLambda}_2$-module. Therefore $\widehat{B}_L \simeq \widehat{\varLambda}_2/(\varDelta_L(1+T_1,1+T_2))$ (cf.\ \cite{Was} p.299 Example (2)). Since $\varDelta_L(1,1)=\pm \ell_{12}$ by the Torres conditions \cite{Tor53}, we have $0 \neq \varDelta_L(1+T_1,1+T_2) \in \widehat{\varLambda}_2^{\times} \cap \mathrm{Ann}_{\widehat{\varLambda}_2} \widehat{B}_L$ if $v_p(\ell_{12})=0$. Then $\widehat{B}_L=0$. 
\end{proof}

\begin{remark}{\em 
$Y_{\widetilde{k}}=0$ for the analogous situation (cf.\ \S3). 
This proposition is also a consequence of Theorem 1.1.6 of \cite{HMM}. 
}\end{remark}

A pro-$p$ version of Greenberg type problem has trivial answer as follows. 

\begin{proposition}
$Y_L \otimes_{\varLambda_r} \widehat{\varLambda}_r =0$. 
\end{proposition}

\begin{proof}
By Levine's theorem \cite{Lev83}, $Y_L \otimes_{\varLambda_r} \widehat{\varLambda}_r \simeq (Y_L \otimes_{\varLambda_r} \varLambda_{r,J}) \otimes_{\varLambda_{r,J}} \widehat{\varLambda}_r \simeq 0 \otimes_{\varLambda_{r,J}} \widehat{\varLambda}_r \simeq 0$. 
\end{proof}

This proposition implies that the Iwasawa invariants $\lambda_{L,\mathbf{z}}$, $\mu_{L,\mathbf{z}}$ and $\nu_{L,\mathbf{z}}$ are related to only the structure of the module $M_L \otimes_{\varLambda_r} \widehat{\varLambda}_r$ genetared by meridianal elements.

\vspace*{20pt}
\noindent\textbf{Acknowledgement.} 
The first author was partially supported by a grant (No.\ 10801021/a010402) of NSFC, and he was financially supported by Professor Akio Kawauchi during his stay in Japan. The second author was partially supported by Grant-in-Aid for Young Scientists (B) (No.\ 22740010). 


\vfill
{\small 
\textsc{Teruhisa Kadokami}: 
Department of Mathematics, East China Normal University, Dongchuan-lu 500, Shanghai 200241, China. \texttt{mshj@math.ecnu.edu.cn}
\\[-3mm]%

\textsc{Yasushi Mizusawa}: 
Department of Mathematics, Nagoya Institute of Technology, Gokiso, Showa, Nagoya 466-8555, Japan. \texttt{mizusawa.yasushi@nitech.ac.jp}
}
\end{document}